\documentclass[12pt]{amsart}
\usepackage{amssymb}%
\usepackage{graphicx}%
\usepackage{amsthm}%
\usepackage{amsmath}%
\usepackage{amsfonts} 
\usepackage{latexsym}%
\usepackage{epsfig}%
\usepackage{epic}%
\usepackage{eufrak}%
\usepackage{amscd}%
\usepackage{color}%
\usepackage[linktoc=all, pagebackref, hyperindex]{hyperref}%
\hypersetup{colorlinks,  citecolor=blue,  filecolor=blue,  linkcolor=red,  urlcolor=black}
\theoremstyle{plain}
  \newtheorem{thm}{Theorem}[section]
  \newtheorem{prop}[thm]{Proposition}
  \newtheorem{lem}[thm]{Lemma}
  \newtheorem{cor}[thm]{Corollary}
  
\theoremstyle{definition}
  \newtheorem{dfn}[thm]{Definition}
  \newtheorem{exmp}[thm]{Example}

\theoremstyle{remark}
  \newtheorem{rem}[thm]{Remark}
    {\begin{list}
        {\noindent\makebox[0mm][r]{\rm(\arabic{enumi})}}
        {\leftmargin=5.5ex \usecounter{enumi}}
    }
    {\end{list}}

\newcommand{\Ann}{\operatorname{Ann}}

\newcommand{\Ass}{\operatorname{Ass}}

\newcommand{\depth}{\operatorname{depth}}

\newcommand{\Ext}{\operatorname{Ext}}

\newcommand{\Hom}{\operatorname{Hom}}
\newcommand{\height}{\operatorname{height}}
\newcommand{\cd}{\operatorname{cd}}

\newcommand{\Rad}{\operatorname{Rad}}

\newcommand{\Spec}{\operatorname{Spec}}
\newcommand{\Supp}{\operatorname{Supp}}

\newcommand{\V}{\operatorname{V}}

\newcommand{\grade}{\operatorname{grade}}
\newcommand{\fgrade}{\operatorname{fgrade}}

\def\1{{\mathbf 1}}

\def\fa{{\mathfrak a}}
\def\fb{{\mathfrak b}}

\def\fm{{\mathfrak m}}

\def\fp{{\mathfrak p}}
\def\fq{{\mathfrak q}}

\def\<{{\langle}}
\def\>{{\rangle}}

\definecolor{myc}{cmyk}{0.3,0.5,1,0}
\newcommand{\excise}[1]{}
\begin{document}

\title[Vanishing of ext-functors and Faltings' Annihilator Theorem]{Vanishing of ext-functors and Faltings' Annihilator Theorem for relative Cohen-Macaulay modules}

\address{Department of Mathematics, Payame Noor University, Tehran, 19395-3697, Iran.} 
\email{khahmadi@pnu.ac.ir}

\author[Ahmadi Amoli, Mast Zohouri and Faramarzi]{Kh. Ahmadi Amoli, M. Mast Zohouri and S. O. Faramarzi}

\address{Department of Mathematics, Payame Noor University, Tehran, 19395-3697, Iran.} 
\email{m.mast.zohouri@gmail.com}

\address{Department of Mathematics, Payame Noor University, Tehran, 19395-3697, Iran.}
\email{s.o.faramarzi@gmail.com}


\begin{abstract}
Let $R$ be a commutative Noetherian ring, $\fa$, $\fb$ be two ideals of $R$, and $M$ be a finite $R$-module. Some vanishing and relative Cohen-Macaulayness of Ext-functors based on RCMF modules are studied. It is shown that over an arbitrary Noetherian ring, the Faltings' Annihilator Theorem holds for any relative Cohen-Macaulay module $M$ w.r.t $\fa$ with $\Supp(M/\fa M)\setminus\V(\fb)\neq\phi$. A number of new results have been derived of the finiteness dimension $f_{\fa}(M)$ of $M$ relative to $\fa$. Also, some applications of relative Cohen-Macaulay modules w.r.t $\fa$ have been shown.
\end{abstract}

\maketitle
\vspace{-2ex}%
\setcounter{tocdepth}{1}%

\section{Introduction}
Throughout this paper, $R$ is a commutative Noetherian ring with identity of positive Krull dimension $n$, $\fa$ and $\fb$ two ideals of $R$ and $M$ is a finite $R$-module of positive dimension $d$. We denote by $H^{i}_{\fa}(M)$ the ith local cohomology module of $M$ supported in $\fa$. 
Before formulating the results of the paper we recall some  related concepts. Based on \cite[Definition 2.2]{RZ}, a finite $R$-module $M$ is called {\it relative Cohen-Macaulay} w.r.t $\fa$ if there is precisely one non-vanishing local cohomology module of $M$ w.r.t $\fa$. Clearly this is the case if and only if $\grade(\fa,M)=\cd(\fa,M)$, where $\cd(\fa,M)$ is the largest integer $i$ for which $H^{i}_{\fa}(M)\neq 0$ and $\grade(\fa,M)$ is the least integer $i$ for which $H^{i}_{\fa}(M)\neq 0$. This concept has a connection with a notion which has been studied under the title of {\it cohomologically complete intersection ideals} in \cite{HSch1}. The present authors studied some properties of relative Cohen-Macaulay rings and modules in \cite{MAF} and \cite{M}.\\
In view of \cite[Definition 2.1]{ASN}, the increasing filtration $\mathcal{M}=\{M_{i}\}^{c}_{i=0}$ of submodules of $M$ with $c=\cd(\fa,M)$ is called the {\it cohomological dimension filtration} of $M$ if for all integer $0\leq i\leq c$, $M_{i}$ is the largest submodule of $M$ such that $\cd(\fa,M_{i})\leq i$. Recently, in \cite{MAF}, we have defined and studied the concept of relative Cohen-Macaulay filtered modules w.r.t $\fa$ (abbreviated as RCMF modules w.r.t $\fa$) as follows. Let $\mathcal{M}=\{M_{i}\}^{c}_{i=0}$ be the cohomological dimension filtration of submodules of $M$, where $c=\cd(\fa,M)$. Then $M$ is said to be a {\it relative Cohen-Macaulay filtered module} ({\it relative sequentially Cohen-Macaulay module}) w.r.t $\fa$, whenever $\mathcal{M}_{i}={M_{i}}/{M_{i-1}}$ is either zero or an $i$-cohomological dimensional relative Cohen-Macaulay module w.r.t $\fa$ for all $1\leq i\leq c$ (see \cite[Definition 2.3]{MAF}).
\\
Recall that the finiteness dimension $f_{\fa}(M)$ of $M$ relative to $\fa$ is defined as
$$
f_{\fa}(M):=\inf\{i\in\mathbb{N}_0: H^{i}_{\fa}(M)\ \text{is not finitely generated}\},
$$
if there exists and $+\infty$ otherwise. Another formulation is the $\fb$-finiteness dimension of $M$ relative to $\fa$ which is defined as 
$$
f_{\fa}^{\fb}(M):=\inf\{i\in\mathbb{N}_0: \fb\nsubseteq\Rad(0:_R H^{i}_{\fa}(M))\},
$$
which is either a positive integer or $+\infty$.
\\
Moreover, the $\fb$-minimum $\fa$-adjusted depth $\lambda_{\fa}^{\fb}(M)$ of $M$ is defined as
$$
\lambda_{\fa}^{\fb}(M):=\inf\{\depth M_{\fp}+\height(\fa+\fp)/\fp: \fp\in\Spec(R)\setminus\V(\fb)\},
$$
which is either a positive integer or $+\infty$ where $\V(\fb)$ denotes the set of prime ideals containing $\fb$.
\\
The paper is divided into 3 sections. In section 2, we study vanishing and relative Cohen- Macaulayness of Ext-functors for RCMF modules with respect to an ideal. In fact, in \cite[Theorem 1.4]{H}, the Peskine's characterization of sequentially Cohen-Macaulay modules is presented in terms of Ext-groups in the graded case. Also, in \cite[Definition 2.10]{MAF}, we have defined cohomological canonical module of $M$ as $\Hom_R(H^{\cd(\fa,M)}_{\fa}(M),E_R(k))$ an denoted it by $K(M)$. As a main result of this paper we prove the following result.\\
Let $(R,\fm)$ be a relative Cohen-Macaulay local ring w.r.t $\fa$ with $\cd(\fa,R)=c$ and $K(R)$ be its cohomological canonical module. Assume that $M$ is an RCMF module w.r.t $\fa$ with the filtration in Lemma~\ref{l1} and $c_i=\cd(\fa,M_i/M_{i-1})$ for all $1\leq i\leq r$. Then $\Ext^{c-c_i}_{R}(M,K(R))\cong\Ext^{c-c_i}_{R}(M_i/M_{i-1},K(R))$ is relative Cohen-Macaulay w.r.t $\fa$ of cohomological dimension $c_i$ for some $1\leq i\leq r$ and $\Ext^{j}_{R}(M,K(R))=0$ if $j\notin\{c-c_1,\ldots,c-c_r\}$. Consequently, in Corollary~\ref{c1}, we partially show a similar result for RCMF $R$-modules w.r.t $\fa$, where $R$ is an arbitrary commutative Noetherian ring. More precisely, we show that if $(R,\fm)$ is a local ring which is relative Cohen-Macaulay w.r.t $\fa$ with $\cd(\fa,R)=c$ and $M$ is an RCMF module w.r.t $\fa$, then the modules $\Ext^{c-i}_{R}(M,K(R))$ are either zero for all $0\leq i\leq\cd(\fa,M)$ or relative Cohen-Macaulay w.r.t $\fa$ of cohomological dimension $i$. 
\\
Section 3 is devoted to examine the behavior of the notion $f_{\fa}(M)$ more closely as the annihilators of local cohomology modules is of crucial importance for researches in local cohomology. In analogy with the notation $\fa(M)={\fa}_0(M)\ldots {\fa}_{d-1}(M)$ in a local ring $(R,\fm)$, defined in \cite{C} and \cite{SCH}, where ${\fa}_i(M)$ is the annihilator of $H^{i}_{\fm}(M)$ for all $0\leq i\leq d-1$, we define the following notation for an $R$-module $M$ and an ideal $\fa$ of $R$ with $h:=\height_M(\fa)$, where $R$ is not necessarily local. We set
$$
{\fa}'(M):={\fa}'_{0}(M)\ldots{\fa}'_{h-1}(M){\fa}'_{h+1}(M)\ldots{\fa}'_d(M),
$$
where ${\fa}'_{i}(M)=\Ann_R(H^{i}_{\fa}(M))$ for all $0\leq i\leq d$ and $i\neq h$ and give some results concerning this notation.
\\
One question about the annihilation of local cohomology modules is that, which ideals annihilate the local cohomology modules, and the classical theorem of this subject is Faltings' Annihilator Theorem (see \cite{F}). Faltings' Annihilator Theorem states that, if $R$ is a homomorphic image of a regular ring or if $R$ has a dualizing complex, then for every choice of ideals $\fa$ and $\fb$ of $R$, $\lambda_{\fa}^{\fb}(M)=f_{\fa}^{\fb}(M)$. There are some improvement of the conditions of this theorem (see for example \cite{BSH}). In Theorem~\ref{t1}, as one of our main result of this paper, we shall show that over an arbitrary Noetherian ring $R$, the Faltings' Annihilator Theorem of local cohomology modules holds for any relative Cohen-Macaulay $R$-module $M$ w.r.t $\fa$ whenever $\Supp(M/\fa M)\setminus\V(\fb)$ is not empty. As a particular case of this main result we conclude that on a Noetherian local ring $(R,\mathfrak{m})$, if $M$ is a Cohen-Macaulay $R$-module, then $f_{\mathfrak{m}}^{R}(M)=\lambda_{\mathfrak{m}}^{R}(M)$.
As it is mentioned in Remark~\ref{r3}, if $M$ is a relative Cohen-Macaulay $R$-module w.r.t $\fa$, $f_{\fa}(M)$ is the largest number in which $H^{f_{\fa}(M)}_{\fa}(M)$ is not Artinian.
In contrast, Proposition~\ref{p4} presents an Artinian submodule of such modules. Corollary~\ref{c7} determines the behavior of the finiteness dimension of modules respect to non-zerodivisors, where the underlying module is relative Cohen-Macaulay w.r.t $\fa$. Some characterizations of $f_{\fa}(-)$ are provided via ring homomorphisms (Corollaries~\ref{c6} and~\ref{c2}). Furthermore, we present two classes of modules, ``multiplication" and ``semidualizing" modules that satisfy $f_{\fa}(M)=f_{\fa}(R)$ (Corollaries~\ref{c4} and ~\ref{c5}).
It is well-known that if $R$ is a complete Cohen-Macaulay local ring, then $w_R=\Hom_R(H^{n}_{\fm}(R),E(R/\fm))$ is a canonical module of $R$.
We end this section by Corollary~\ref{c3} which shows that for a relative Cohen-Macaulay local ring $R$ w.r.t $\fa$, $f_{\fa}(w_R)=f_{\fa}(R)$ where $w_R$ denotes the canonical module of $R$.
\\
For unexplained notation and terminology about local cohomology modules, we refer the reader to \cite{BSH}.

\section{Vanishing and relative Cohen-Macaulayness of Ext-functors based on RCMF modules}\label{3}
This section is devoted to deal with vanishing and relative Cohen-Macaulayness of Ext-functors for RCMF modules w.r.t $\fa$. We begin by recalling the following.

\begin{dfn}(see \cite[Definition 2.3]{MAF}) \label{d1} 
Let $M$ be a finite $R$-module and $\mathcal{M}=\{M_{i}\}^{c}_{i=0}$ be the cohomological dimension filtration of submodules of $M$, where $c=\cd(\fa,M)$. $M$ is called a {\it relative Cohen-Macaulay filtered module} ({\it relative sequentially Cohen-Macaulay module}) w.r.t $\fa$, whenever $\mathcal{M}_{i}={M_{i}}/{M_{i-1}}$ is either zero or an $i$-cohomological dimensional relative Cohen-Macaulay module w.r.t $\fa$ for all $1\leq i\leq c$. Let us abbreviate this notion by RCMF.
\end{dfn}

\medskip
\begin{dfn} (see \cite[Definition 2.10]{MAF})\label{def} 
Let $(R,\fm,k)$ be a local ring and $M$ be a finite $R$-module with $\cd(\fa,M)=c$. For $i \neq c$, the {\it ith cohomological deficiency module} of $M$ is defined by $$K^{i}_{\fa}(M):=D_{R}(H^{i}_{\fa}(M))=\Hom_{R}(H^{i}_{\fa}(M),E_{R}(k)).$$
The module $K(M):=K^{c}_{\fa}(M)$ is called the {\it cohomological canonical module} of $M$. Note that $K^{i}_{\fa}(M)=0$ for all $i<0$ or $i>c$.
\end{dfn}

\medskip
\begin{lem} (see \cite[Theorem 4.3]{RZ}) \label{l2}
Let $(R,\fm)$ be a relative Cohen-Macaulay local ring w.r.t $\fa$ with $\cd(\fa,R)=c$ and $K(R)$ be its canonical module. For all $0\leq t\leq c$ and all relative Cohen-Macaulay $R$-module $M$  w.r.t $\fa$ with $\cd(\fa,M)=t$, we have  
\begin{itemize}
\item[(i)]
$\Ext^{j}_{R}(M,K(R))=0$ for all $j\neq c-t$.
\item[(ii)]
$\Ext^{c-t}_{R}(M,K(R))$ is a relative Cohen-Macaulay module w.r.t $\fa$ of cohomological dimension $t$.
\end{itemize}
\end{lem}

\medskip
The next lemma turns out directly from Definition~\ref{d1}.
\begin{lem}\label{l1}
Let $M$ be an RCMF $R$-module w.r.t $\fa$ with its filtration as $0=M_0\subset M_1\subset\ldots\subset M_r=M$ where $r=\cd(\fa,M)$. Then for each $0\leq i\leq r$, the module $M/M_i$ is RCMF w.r.t $\fa$ with the filtration $0=M_i/M_i\subset M_{i+1}/M_i\subset\ldots\subset M_r/M_i=M/M_i$.
\end{lem}

\medskip
In order to prove Corollary~\ref{c1}, we prove the following.
\begin{thm}
Let $(R,\fm)$ be a relative Cohen-Macaulay local ring w.r.t $\fa$ with $\cd(\fa,R)=c$ and $K(R)$ be its cohomological canonical module. Assume that $M$ is an RCMF module w.r.t $\fa$ with the filtration in Lemma~\ref{l1} and $c_i=\cd(\fa,M_i/M_{i-1})$ for all $1\leq i\leq r$. Then $\Ext^{c-c_i}_{R}(M,K(R))\cong\Ext^{c-c_i}_{R}(M_i/M_{i-1},K(R))$ is relative Cohen-Macaulay w.r.t $\fa$ of cohomological dimension $c_i$ for some $1\leq i\leq r$ and $\Ext^{j}_{R}(M,K(R))=0$ if $j\notin\{c-c_1,\ldots,c-c_r\}$.
\end{thm}
\begin{proof}
The proof is by induction on $r$. For $r=1$, $M=M_1/M_0$ is relative Cohen-Macaulay w.r.t $\fa$. Hence, by Lemma~\ref{l2}, $\Ext^{j}_{R}(M,K(R))=0$ for all $j\neq c-1$ and $\Ext^{c-1}_{R}(M,K(R))$ is relative Cohen-Macaulay w.r.t $\fa$ of cohomological dimension one. (Note that $c_1=\cd(\fa,M_1/M_0)$ is equal to one)
Now, assume that $r>1$ and, inductively, the assertion is true for $r-1$. From the short exact sequence $$0\longrightarrow M_1\longrightarrow M\longrightarrow M/M_1\longrightarrow 0,$$ we get the following long exact sequence

\[\begin{array}{ll}
\ldots & \longrightarrow\Ext^{j}_{R}(M/M_1,w_R)\longrightarrow\Ext^{j}_{R}(M,w_R)
\longrightarrow\Ext^{j}_{R}(M_1,w_R) \\ & \longrightarrow\Ext^{j+1}_{R}(M/M_1,w_R)\longrightarrow\ldots
\end{array}\]

By Lemma~\ref{l2}, $\Ext^{j}_{R}(M_1,K(R))=0$ for all $j\neq c-c_1$ and $\Ext^{c-c_1}_{R}(M_1,K(R))$ is relative Cohen-Macaulay w.r.t $\fa$ with cohomological dimension $c_1$. Thus we get the following exact sequence
\[\begin{array}{ll}
0 & \longrightarrow\Ext^{c-c_1}_{R}(M/M_1,K(R))\longrightarrow\Ext^{c-c_1}_{R}(M,K(R))
\longrightarrow\Ext^{c-c_1}_{R}(M_1,K(R)) \\ & \longrightarrow\Ext^{c-c_1+1}_{R}(M/M_1,K(R))\longrightarrow\Ext^{c-c_1+1}_{R}(M,K(R))\longrightarrow 0,
\end{array}\]
and the isomorphisms $\Ext^{j}_{R}(M/M_1,K(R))\cong\Ext^{j}_{R}(M,K(R))$ for all $j\neq c-c_1,c-c_1+1$. On the other hand, $M/M_1$ is RCMF module w.r.t $\fa$ by Lemma~\ref{l1} and $M_1/M_1\subset M_2/M_1\subset\ldots\subset M_r/M_1=M/M_1$ is the relative Cohen-Macaulay filtration w.r.t $\fa$ of length $r-1$. Therefore by inductive hypothesis and Lemma~\ref{l1}, $\Ext^{j}_{R}(M/M_1,K(R))=0$ for $j\notin\{c-c_2,\ldots,c-c_r\}$. (Note that $\cd(\fa,M_i/M_1/M_{i-1}/M_1)$ is equal to $c_i$ for all $2\leq i\leq r$.) This implies that $\Ext^{c-c_1}_{R}(M/M_1,K(R))=0$ and $\Ext^{c-c_1+1}_{R}(M/M_1,K(R))=0$. So by the last exact sequence we have $\Ext^{c-c_1}_{R}(M,w_R)\cong\Ext^{c-c_1}_{R}(M_1,K(R))$ is relative Cohen-Macaulay w.r.t $\fa$ of cohomological dimension $c_1$ and $\Ext^{c-c_1+1}_{R}(M,K(R))=0$. Thus we get the isomorphism $\Ext^{j}_{R}(M/M_1,K(R))\cong\Ext^{j}_{R}(M,K(R))$ for all $j\neq c-c_1$. 
This implies that $\Ext^{j}_{R}(M,K(R))=0$ for all $j\notin\{c-c_1,c-c_2,\ldots,c-c_r\}$. Now, by this procedure, considering the filtration $M_i/M_i\subset M_{i+1}/M_i\subset\ldots\subset M/M_i$ for $M/M_i$ we get the assertion by induction.
\end{proof}

\medskip
\begin{cor}\label{c1}
Let $(R,\fm)$ be a relative Cohen-Macaulay local ring w.r.t $\fa$ with $\cd(\fa,R)=c$ and $K(R)$ be its cohomological canonical module. If $M$ is an RCMF $R$-module w.r.t $\fa$, then for all $0\leq i\leq\cd(\fa,M)$, the modules $\Ext^{c-i}_{R}(M,K(R))$ are either zero or relative Cohen-Macaulay w.r.t $\fa$ of cohomological dimension $i$.
\end{cor}

\medskip
If we prove the converse inclusion of Corollary~\ref{c1}, we can conclude that the finite direct sum of RCMF modules is an RCMF module. There comes naturally a question whether the converse is true or not. Of course, it is easy to show that the direct sum of two relative Cohen-Macaulay modules of the same cohomological dimension is again a relative Cohen-Macaulay module.

\section{The finiteness dimension of modules and relative Cohen-Macaulayness}
In this section, some subjects such as the finiteness dimension $f_{\fa}(M)$ of $M$ relative to $\fa$ and the $\fb$-minimum $\fa$-adjusted depth $\lambda_{\fa}^{\fb}(M)$ of $M$ are studied. Note that we do not assume that $\fb\subseteq\fa$ in general. Furthermore, we introduce the ideal ${\fa}'(M)$ to describe arisen problems. We also prove that Faltings' Annihilator Theorem holds for relative Cohen-Macaulay $R$-modules. An important tool in this section is relative Cohen-Macaulayness w.r.t an ideal. We begin by defining the following notation in comparability with the notation ${\fa}(M)$ used in \cite{C} and \cite{SCH}.

\begin{dfn}
Let $M$ be an $R$-module and $\fa$ an ideal of $R$ with $h:=\height_M(\fa)$. We set
$$
{\fa}'(M):={\fa}'_{0}(M)\ldots{\fa}'_{h-1}(M){\fa}'_{h+1}(M)\ldots{\fa}'_d(M),
$$
where ${\fa}'_{i}(M)=\Ann_R(H^{i}_{\fa}(M))$ for all $0\leq i\leq d$ and $i\neq h$.
\end{dfn}

\medskip
\begin{rem}
Note that if $R$ is a relative Cohen-Macaulay local ring w.r.t $\fa$, then by the above notation we have ${\fa}'(R)=R$. Thus one deduces immediately that $\dim(R/{\fa}'(R))=-1$. In addition, it follows that $f_{{\fa}'(R)}(R)=\infty$.
\end{rem}

\medskip
The following corollary presents an interesting relation to the annihilators of components of the filtration of an RCMF module. 
\begin{cor}
Let $M$ be an RCMF module w.r.t $\fa$ and $\mathcal{M}=\{M_i\}^{c}_{i=0}$ be its cohomological dimension filtration, where $c=cd(\fa,M)$. Then
$$
\fa'_i(M)=\fa'_i(M_i)=\fa'_i(\mathcal{M}_i)
$$
for all $0\leq i\leq c$ with $i\neq\height_M \fa$.
\end{cor}
\begin{proof}
Use \cite[Proposition 2.12]{MAF}.
\end{proof}

\medskip
\begin{cor}
Let $M$ be a relative Cohen-Macaulay $R$-module w.r.t $\fa$ with $M\neq\fa M$. Then $f_{\fa}^{\fa'(M)}(M)=\lambda_{\fa}^{\fa'(M)}(M)$.
\end{cor}
\begin{proof}
As $\fa'(M)=R$, we have $f_{\fa}^{R}(M)=\grade(\fa,M)=\lambda_{\fa}^{R}(M)$ by \cite[Exercises 9.1.9 and 9.3.3]{BSH} as desired.
\end{proof}

\medskip
\begin{rem}\label{r2}
For any finite $R$-module $M$ such that $\height_M(\fa)>0$, we have the inequality $f_{\fa}(M)\leq\height_M(\fa)$ (see \cite[Remark 2.7(i)]{ADT}). It is clear that if $M$ is relative Cohen-Macaulay $R$-module w.r.t $\fa$, then $f_{\fa}(M)=\height_M(\fa)$. But there are some examples in which $f_{\fa}(M)<\height_M(\fa)$. Let $k$ be a field and $R=k[X,Y^2,XY,Y^3]$. The height of the ideal $\fa=(X,Y^2,XY,Y^3)R[Z]+Z R[Z]$ of the ring of polynomials $R[Z]$ is equal to $3$ whiles the finiteness dimension $f_{\fa}(R[Z])$ is $2$ (see \cite[Example 9.5.3]{BSH}).
\end{rem}

\medskip
In the following result, we provide an upper bound for $\lambda_{\fa}^{\fb}(M)$.
\begin{prop}\label{p2}
Let $M$ be a relative Cohen-Macaulay $R$-module w.r.t $\fa$ such that $\Supp(M/\fa M)\setminus\V(\fb)\neq\phi$. Then 
$$
\lambda_{\fa}^{\fb}(M)\leq\height_{M}(\fa).
$$
\end{prop}
\begin{proof}
Let $\fp\in\Supp(M/\fa M)\setminus\V(\fb)$ and $\fq$ be a minimal prime ideal of $\fa+\Ann_R(M)$ such that $\fq\subseteq\fp$. Then 
$$
H^{i}_{\fa}(M)_{\fq}\cong H^{i}_{(\fa+\Ann_R M)R_{\fq}}(M_{\fq})\cong H^{i}_{\fq R_{\fq}}(M_{\fq})
$$
for all $i\geq 0$. By \cite[Theorem 6.1.4]{BSH} and since $H^{\depth M_{\fq}}_{\fq R_{\fq}}(M_{\fq})$ is not zero, we conclude that 
$$
\depth M_{\fq}=\dim M_{\fq}.
$$
But $\dim M_{\fq}=\height_M(\fa)$, so that 
$$
\depth M_{\fq}=\height_M(\fa).
$$
Now, by definition, $\lambda_{\fa}^{\fb}(M)\leq\depth M_{\fq}=\height_M(\fa)$ as required.
\end{proof}

\medskip
Now, we are in position, to bring our main result of this section which shows that the Faltings' Annihilator Theorem holds for relative Cohen-Macaulay modules.
\begin{thm}\label{t1}
Let $R$ be a Noetherian ring and $M$ be a relative Cohen-Macaulay $R$-module w.r.t $\fa$ such that $\Supp(M/\fa M)\setminus\V(\fb)\neq\phi$. Then 
$$
f_{\fa}^{\fb}(M)=\lambda_{\fa}^{\fb}(M).
$$
\end{thm}
\begin{proof}
As $M$ is relative Cohen-Macaulay w.r.t $\fa$, it is easy to see that $\height_M(\fa)=\cd(\fa,M)=f_{\fa}^{\fb}(M)$. Therefore by \cite[Theorem 9.3.7]{BSH} and Proposition~\ref{p2}, we get 
$$
\height_M(\fa)\leq\lambda_{\fa}^{\fb}(M)\leq\height_M(\fa)
$$
as desired.
\end{proof}

\medskip
For a particular case of Theorem~\ref{t1}, we obtain the following result.
\begin{cor}
Let $(R,\fm)$ be a Noetherian local ring. If $M$ is a Cohen-Macaulay $R$-module, then $f_{\fm}^{R}(M)=\lambda_{\fm}^{R}(M)$.
\end{cor}
\begin{proof}
It is an immediate consequence of Theorem~\ref{t1}, since $M$ is relative Cohen-Macaulay w.r.t $\fm$. 
\end{proof}

\medskip
\begin{rem}\label{r3}
Let $I$ be an ideal of $R$. Recall that a sequence $a_1,\ldots,a_n$ of elements of $R$ is an $I$-filter regular $M$-sequence if $a_i\notin\fp$ for all $\fp\in\Ass(M/(a_1,\ldots,a_{i-1})M)\setminus\V(I)$ ($i=1,\ldots,n$) and $\Supp(M/(a_1,\ldots,a_n)M)\setminus\V(I)\neq\phi$. It is well-known that the length of any maximal $I$-filter regular $M$-sequence contained in $\fa$ which is denoted by $I-\fgrade_M(\fa)$ is the least integer $i\geq 0$ such that $\Supp(H^{i}_{\fa}(M))\nsubseteq\V(I)$, provided $\Supp(M/\fa M)\setminus\V(I)\neq\phi$ (see \cite[Theorem 1.13]{A} and P. 38 of \cite{T}). It is easy to see that if $(R,\fm)$ is a local ring and $\fa$ is an ideal of $R$ with $\Supp(M/\fa M)\setminus\{\fm\}\neq\phi$, then
$$
\fm-\fgrade_M(\fa)=\min\{i\in\mathbb{N}_0\mid H^{i}_{\fa}(M)\ \text{is not Artinian} \}.
$$
In the above situation, if $M$ is a relative Cohen-Macaulay $R$-module w.r.t $\fa$, then since $H^{f_{\fa}(M)}_{\fa}(M)\neq 0$, we get 
$$
f_{\fa}(M)=\height_M(\fa)=\fm-\fgrade_M(\fa).
$$
Therefore $H^{f_{\fa}(M)}_{\fa}(M)$ is not Artinian. With the notion $\fq_{\fa}(M)$ used in \cite{Ha}, since the top local cohomology module $H^{\dim M}_{\fa}(M)$ is Artinian, we have $q_{\fa}(M)<\dim M$. In particular, if $H^{i}_{\fa}(M)=0$ for all $i\neq\height_M(\fa)$, then $q_{\fa}(M)=f_{\fa}(M)$. 
\end{rem}

\medskip
Although, the top local cohomology modules are almost always non-Artinian as we have seen in Remark~\ref{r3}, we provide conditions to extract an Artinian submodule of such modules as follows.
\begin{prop}\label{p4}
Let $(R,\fm)$ be a local ring and $M$ be a relative Cohen-Macaulay $R$-module w.r.t $\fa$ with $\height_M(\fa)=h$ and $\dim_R(M/\fa M)=1$. Then $H^{0}_{Rx}(H^{h}_{\fa}(M))$ is Artinian $R$-module for some $x\in\fm$.
\end{prop}
\begin{proof}
The assumption $\dim(M/\fa M)=1$ guaranties the existence of $x\in\fm$ such that $\dim M/(\fa+Rx)M=0$. Considering the following exact sequence 
$$
0\longrightarrow H^{1}_{Rx}(H^{h-1}_{\fa}(M))\longrightarrow H^{h}_{\fm}(M)\longrightarrow H^{0}_{Rx}(H^{h}_{\fa}(M))\longrightarrow 0
$$
\cite[Corollary 3.5]{SCH2} and the fact that $H^{h-1}_{\fa}(M)=0$ we deduce that $H^{0}_{Rx}(H^{h}_{\fa}(M))$ is Artinian over $R$.
\end{proof}

\medskip
In \cite[Lemma 2.7]{MNS}, it is proved that $f_{\fa}(M)\leq s+f_{\fa}(M/\underline{x}M)$ where $s$ is a non-negative integer and $\underline{x}=x_1,\ldots,x_s$ is an $M$-regular sequence. We now provide conditions that this inequality becomes an equality. 
\begin{cor}\label{c7}
Let $(R,\fm)$ be a local ring and $M$ be a relative Cohen-Macaulay $R$-module w.r.t $\fa$ with $c:=\height_M \fa$. Assume that $\underline{x}=x_1,\ldots,x_s$ is a regular sequence on both $M$ and $D_R(H^{c}_{\fa}(M))$. Then $f_{\fa}(M)=f_{\fa}(M/\underline{x}M)+s$.
\end{cor}
\begin{proof}
According to definition and \cite[Corollary 2.23]{MAF}, we have $f_{\fa}(M)=\cd(\fa,M)$ and $f_{\fa}(M/\underline{x}M)=\cd(\fa,M/\underline{x}M)$. Now, the assertion follows by \cite[Theorem 2.22]{MAF}.
\end{proof}


\medskip
In the remainder of this section, we present a number of corollaries about the finiteness dimension $f_{\fa}(M)$ of $M$ relative to $\fa$, where $M$ is a relative Cohen-Macaulay $R$-module w.r.t $\fa$. In particular, in two corollaries below, we characterize the behavior of it under some homomorphisms.
\begin{cor}\label{c6}
Assume that for all integers $k\in\mathbb{N}$ there is a ring homomorphism $\phi_k:R\longrightarrow R$ such that the inverse system of ideals $\{\phi_k(\fa)R\}_{k\in\mathbb{N}}$ is cofinal with $\{{\fa}^k\}_{k\in\mathbb{N}}$. If $R$ is relative Cohen-Macaulay w.r.t $\fa$, then $f_{\fa}(R)=f_{\phi_k(\fa)R}(R)$ for all $k\in\mathbb{N}$.
\end{cor}
\begin{proof}
First, as the inverse system of ideals $\{\phi_k(\fa)R\}_{k\in\mathbb{N}}$ is cofinal with $\{{\fa}^k\}_{k\in\mathbb{N}}$, for all $i$ there exists $j$ such that ${\fa}^j\subseteq\phi_i(\fa)R\subseteq\fa$. Thus $\grade(\phi_k(\fa)R,R)=\grade(\fa,R)$ for all $k\in\mathbb{N}$. On the other hand, $\cd(\phi_k(\fa)R,R)=\cd(\fa,R)$ for all $k\in\mathbb{N}$ by \cite[lemma 2.1]{DNT}. Thus $R$ is relative Cohen-Macaulay w.r.t $\phi_k(\fa)R$ for all $k\in\mathbb{N}$. Now, Remark~\ref{r2} completes the proof.
\end{proof}

\medskip
\begin{exmp}
For any Noetherian ring $R$ of positive characteristic, the preceding corollary is true if we take $\phi_k:R\longrightarrow R$ as the Frobenius map.
\end{exmp}

\medskip
\begin{cor}\label{c2}
Let $f:R\longrightarrow R'$ be a faithfully flat homomorphism of Noetherian rings. If $R$ is relative Cohen-Macaulay w.r.t $\fa$, then $f_{\fa}(R)=f_{\fa R'}(R')$. 
\end{cor}
\begin{proof}
By \cite[Proposition 3.7]{MAF}, $R'$ is a relative Cohen-Macaulay ring w.r.t $\fa R'$. Now, apply Remark~\ref{r2} to end the proof.
\end{proof}

To see examples for Corollary~\ref{c2}, we refer the reader to \cite[Examples 3.9 and 3.10]{MAF}.

\medskip
Now, we remind two classes of $R$-modules for which $f_{\fa}(M)=f_{\fa}(R)$.

\begin{dfn}
(see \cite{ES}) An $R$-module $M$ is called {\it multiplication} if for every submodule $N$ of $M$ there exists an ideal $\fa$ of $R$ such that $N=\fa M$. 
\end{dfn}

\medskip
\begin{cor}\label{c4}
Let $R$ be a relative Cohen-Macaulay ring w.r.t $\fa$ and $M$ be a faithful multiplication $R$-module. Then $f_{\fa}(M)=f_{\fa}(R)$.
\end{cor}
\begin{proof}
Apply \cite[Corollary 3.2]{MAF} and Remark~\ref{r2}.
\end{proof}

\medskip 
\begin{dfn}(see \cite{FO} and \cite{V})
An $R$-module $M$ is {\it semidualizing} if it satisfies the following:
\begin{itemize}
\item[(i)] $M$ is finitely generated,
\item[(ii)] The homotopy map ${\chi}^{R}_{M}: R\longrightarrow \Hom_{R}(M,M)$, defined by $r\mapsto [s\mapsto rs]$, is an isomorphism, and
\item[(iii)] $\Ext^{i}_{R}(M,M)=0$ for all $ i>0.$
\end{itemize}
\end{dfn}

\medskip
\begin{cor}\label{c5}
Let $R$ be a relative Cohen-Macaulay ring w.r.t $\fa$ and $M$ be a semidualizing $R$-module. Then $f_{\fa}(M)=f_{\fa}(R)$.
\end{cor}
\begin{proof}
Apply \cite[Corollary 3.4]{MAF} and Remark~\ref{r2}.
\end{proof}

\medskip
At the end, we obtain a result on canonical modules.
\begin{cor}\label{c3}
Let $(R,\fm)$ be a relative Cohen-Macaulay local ring w.r.t $\fa$ and $w_R$ be its canonical module. Then $f_{\fa}(w_R)=f_{\fa}(R)$.
\end{cor}
\begin{proof}
By \cite[Proposition 3.5]{MAF}, the canonical module $w_R$ is relative Cohen-Macaulay module w.r.t $\fa$. Thus the proof is completed by Remark~\ref{r2}.
\end{proof}

{\bf Acknowledgements.} 
The authors would like to express their thanks to Dr. Raheleh Jafari from Kharazmi University for her useful comments.

\end{document}